\documentclass[12pt]{article}
\usepackage[latin1]{inputenc}
\usepackage[dvips]{graphicx}
\usepackage{latexsym}
\usepackage{amsmath,amssymb,amsfonts,amsthm}

\newcommand{\nrn}{\rightarrow+\infty}
\newcommand{\xrn}{\xrightarrow}
\newcommand{\gam}{\gamma_{n+1}}
\newcommand{\ER}{\mathbb {R}}\newcommand{\EN}{\mathbb {N}}
\newcommand{\PE}{\mathbb {P}}
\newcommand{\ES}{\mathbb{E}}

\newcommand{\guil}{\textquotedblleft}
\newcommand{\psg}{\langle }
\newcommand{\psd}{\rangle }
= 0 cm \evensidemargin = -0.3 cm \marginparwidth = 0 cm
\newtheorem{theorem}{ \textnormal{\bf{T\scriptsize{HEOREM}}}}
\newtheorem{prop}{\textnormal{\bf{P\scriptsize{ROPOSITION}}}}

\newtheorem{lemme}{\textnormal{\bf{L\scriptsize{EMMA}}}}

\theoremstyle{definition}

\theoremstyle{remark}

\newtheorem{Remarque}{\textnormal{\bf{R\scriptsize{EMARK}}}}

\author{Fabien Panloup\footnote{
    Institut de Mathématiques de Toulouse, LSP, Université Paul Sabatier. E-mail: \texttt{fpanloup@insa-toulouse.fr}. Postal adress: Département GMM, INSA Toulouse, 135, Avenue de Rangueil, 31077 TOULOUSE Cedex 4.}}
\title{\textbf{A connection between extreme value theory and long time approximation of SDE's}}

\begin{document}
\maketitle
\begin{abstract} We consider a sequence $(\xi_n)_{n\ge1}$ of $i.i.d.$ random values  living in the domain of attraction of an extreme value distribution. For such sequence, there
exists $(a_n)$ and $(b_n)$, with $a_n>0$ and $b_n\in\ER$ for every
$n\ge 1$, such that the sequence $(X_n)$ defined by $X_n=(\max(\xi_1,\ldots,\xi_n)-b_n)/a_n$
converges in distribution to a non degenerated distribution.\\
In this paper, we show that $(X_n)$ can be viewed as an Euler scheme with
decreasing step of an ergodic Markov process solution to a SDE
with jumps and we derive a functional limit theorem for the sequence
$(X_n)$ from some methods used in the long time numerical approximation of ergodic SDE's.
\end{abstract}
\noindent \textit{AMS classification (2000)}: 60G10, 60G70,60J75, 65D15.
\section{\large{Introduction}}
Let $(\xi_n)_{n\ge1}$ be a sequence of $i.i.d.$ random values with
common distribution function $F$. Set
$M_n:=\max(\xi_1,\ldots,\xi_n)$ and let $G$ denote one of the extreme
values distribution functions:
\begin{align*}
&\Lambda(x)=\exp(-e^{-x}),&x\in\ER,\\
&\Phi_\alpha(x)=\exp(-x^{-\alpha})1_{x>0}&\alpha>0,\\
&\Psi_\alpha(x)=\exp(-(-x)^\alpha)\wedge 1,&\alpha>0.
\end{align*}
One says that $F$ is in the domain of attraction of one of the preceding
extreme values distributions if there exist some sequences $(a_n)$ and
$(b_n)$, with $a_n>0$ and $b_n\in\ER$ for every $n\ge1$, such that
$((M_n-b_n)/a_n)_{n\ge1}$ converges in distribution to $G$ (see
\cite{leadbetter}, \cite{resnick} and Proposition \ref{rappelgne} for
background on extreme value theory).\\
In \cite{lamperti}, Lamperti obtained a functional version of this
result which is analogous to the  Donsker theorem for sums of
independent variables. More precisely, denoting by $(Y^{(n)})$ the
sequence of c\`adl\`ag processes defined by
\begin{equation*}
Y^{(n)}_t=\begin{cases}(M_{[nt]}-b_n)/a_n&\textnormal{if
$t\ge1/n$}\\
(\xi_1-b_n)/a_n&\textnormal{if $0\le t<1/n$},
\end{cases}
\end{equation*}
he proved that $(Y^{(n)})_{n\ge1}$  converges weakly on
$\mathbb{D}(\ER_+,\ER)$ (that denotes the space of càdlàg functions on $\ER_+$ with values in $\ER$ endowed with the Skorokhod topology) to a  c\`adl\`ag  process $Y$ called extremal
process (see also \cite{resnick2}).\\
The aim of this paper is to obtain another functional limit theorem
for the sequence $((M_n-b_n)/a_n)_{n\ge1}$ by  using that
$((M_n-b_n)/a_n)_{n\ge1}$ can be viewed as an approximation of an
Euler scheme with decreasing step of an ergodic Markov process
solution to a SDE with jumps. \\
The motivation is then twofold: on the one hand, we wish to connect  the theory of long time discretization of SDE's and extreme value theory and on the other hand, we want to exhibit another functional asymptotic behavior of extremes of i.i.d random sequences.\\
Set $X_n:=(M_n-b_n)/a_n$. Then, the sequence
$(X_n)$ can be recursively written as follows: $X_1=\xi_1$ and for
every $n\ge1$,
\begin{equation*}
X_{n+1}=\frac{a_n}{a_{n+1}}{X}_n+\frac{b_n-b_{n+1}}{a_{n+1}}+\frac{a_n}{a_{n+1}}
\Big(\frac{\xi_{n+1}}{a_n}-{X}_n-\frac{b_n}{a_n}\Big)_+.
\end{equation*}
For every $n\ge1$, we set $\theta_n:=\inf\{x, F(x)\ge 1-1/n\}$ and
$\gamma_n:=1-F(\theta_n)$. Note that $\gamma_n=1/n$ if $F$ is
continuous and that $(\gamma_n)_{n\ge1}$ is a nonincreasing
sequence. Then, we denote by $(\rho_n)_{n\ge1}$ and
$(\beta_n)_{n\ge1}$ the sequences defined by
$\rho_n:=\frac{a_{n-1}-a_n}{a_n\gamma_n}$ and
$\beta_n:=\frac{b_{n-1}-b_{n}}{a_{n}\gamma_n}$. With these
notations, we have for every $n\ge1$,
\begin{equation}\label{convloischema}
{X}_{n+1}={X}_n+\gamma_{n+1}(\rho_{n+1}
{X}_n+\beta_{n+1})+(1+\rho_{n+1}\gamma_{n+1})\Big(\frac{\xi_{n+1}}{a_n}-{X}_n-\frac{b_n}{a_n}\Big)_+.
\end{equation}
Setting $\Gamma_n:=\sum_{k=1}^n\gamma_k$, we denote by
$({\cal X}^{(n)})_{n\ge 1}$,  the sequence of c\`adl\`ag
processes defined by
\begin{equation}\label{process}
 {\cal X}_t^{(n)}= {X}_{N(n,t)}\qquad
\textnormal{with}\quad N(n,t)=\inf\{k\ge n,
\Gamma_{k+1}-\Gamma_n> t\}.\end{equation} In particular,
${\cal X}^{(n)}_{0}={X}_n$. Finally, we set ${\cal
F}_n=\sigma({X}_1,\ldots,{X}_n)$.\\

\noindent In Equation \eqref{convloischema}, we try to write $X_{n+1}-X_n$ as the increment  of an Euler scheme with step $\gamma_n$ of a SDE (that we identify in the sequel). Then, the stepwise constant process ${\cal X}^{(n)}$ plays the role of a continuous-time version of the Euler scheme starting from time $\Gamma_n$.\\
At this stage, we can observe that two types of terms
appear in the right-hand member of Equation \eqref{convloischema}: the first one is close to the time discretization of  a drift term and the second one looks like 
a positive jump. This heuristic remark will be
clarified in Theorem \ref{maxprincipal}.\\
\section{Main Result}
The main result of this paper is Theorem \ref{maxprincipal}. In this result, we
show under some mild conditions that the sequence
$({\cal X}^{(n)})_{n\ge1}$  converges weakly to a stationary
Markov process for the Skorokhod topology on
$\mathbb{D}(\ER_+,\ER)$. In this way, we first need to recall the
 result  by Gnedenko which characterizes the domain of attraction of each extreme value distribution (see \cite{gnedenko}):
\begin{prop}\label{rappelgne}
Let $(\xi_n)$ be a sequence of $i.i.d.$ real-valued random
variables with distribution function $F$ and set $x_F:=\sup\{x,
F(x)<1\}$. Then, there exists $(a_n)$ and $(b_n)$ such that
$((M_n-b_n)/a_n)$ converges weakly to a random variable with a
non-degenerated distribution function
$G$ if and only if one of the following conditions is fulfilled :\\
\noindent $\bullet$ Type 1 : There exists a positive function $g$
such that
\begin{equation}\label{lambdaproperty}
 \frac{1-F(t+x g(t))}{1-F(t)}\xrightarrow{t\nearrow
x_F}
 -\ln(\Lambda(x))\qquad\forall x\in\ER.
 \end{equation}
\noindent In this case, $G(x)=\Lambda(x)$ for every $x\in \ER$ and the
norming constants $a_n$ and $b_n$ may be chosen as
$a_n=g(\theta_n)$ and $b_n=\theta_n.$\\
$\bullet$ Type $(2,\alpha)$ : $x_F=+\infty$ and for every $x>0$
 $$ \frac{1-F(tx )}{1-F(t)}\xrightarrow{t\rightarrow
 +\infty}-\ln(\Phi_\alpha(x))\qquad \textnormal{where }\alpha >0.$$
\noindent In this case, $G(x)=\Phi_\alpha(x)$ and the
norming constants $a_n$ and $b_n$ may be chosen as $a_n=\theta_n$
and $b_n=0$.\\
 $\bullet$ Type $(3,\alpha)$ : $x_F<+\infty$ and for every $x<0$,
 $$ \frac{1-F(x_F+xt )}{1-F(x_F-t)}\xrightarrow{t\nearrow
 x_F}-\ln(\Psi_\alpha(x))\qquad \textnormal{with }\alpha >0.$$
\noindent In this case, $G(x)=\Psi_\alpha(x)$ and the norming constants
$a_n$ and $b_n$ may be chosen as $a_n=x_F-\theta_n$ and $b_n=x_F$.
\end{prop}

\noindent In the sequel, we will denote by $\nu_G$ the probability associated with $G$ and by $\tau_G$ the limiting function that appears in Proposition \ref{rappelgne}:
\begin{equation}\label{tauf}
\tau_G(x)=\begin{cases}\exp(-x)&\textnormal{if $F$ is of type 1,}\\
x^{-\alpha}1_{x>0}&\textnormal{if $F$ is of type $(2,\alpha)$}\\
(-x)^{\alpha}1_{x\le0}&\textnormal{if $F$ is of type $(3,\alpha)$}
\end{cases}
\end{equation}
We will also denote by ${\cal C}^1_K(D_G)$  the set of ${\cal C}_1$-functions with compact support in $D_G$ with 
\begin{equation*}D_G=\begin{cases}\ER&\textnormal{if $F$ is of type 1,}\\
(0,+\infty)&\textnormal{if $F$ is of type $(2,\alpha)$,}\\
(-\infty,0]&\textnormal{if $F$ is of type $(3,\alpha)$.}
\end{cases}
\end{equation*}
Finally, for a distribution function $F$ of type 1, $(2,\alpha)$ or $(3,\alpha)$, ${\cal C}(F)$ will correspond to the set of sequences
$(a_n,b_n)_{n\ge1}$ (with $a_n>0$ and $b_n\in\ER$) such that, for every sequence $(\xi_n)_{n\ge1}$
of $i.i.d.$ random variables with common distribution function $F$, $\left((\max(\xi_1,\ldots\xi_n)-b_n)/a_n\right)_{n\ge1}$ converges
weakly to $\nu_G$. We recall that if  $(a_n,b_n)_{n\ge1}\in {\cal C}(F)$, then,
$(\tilde{a}_n,\tilde{b}_n)_{n\ge1}\in {\cal C}(F)$ if and only if,
\begin{equation}\label{choixanbn}
a_n\overset{n\rightarrow+\infty}{\sim}\tilde{a}_n\quad\textnormal{and}\quad
\frac{b_n-\tilde{b}_n}{a_n}\xrn{n\nrn}0\qquad \textnormal{(see $e.g.$
\cite{resnick}, Proposition 0.2).}
\end{equation}
 
\noindent Now, we introduce  $\mathbf{H_\Lambda^1}(F)$ and $\mathbf{H_\Lambda^2}(F)$
which are some assumptions that are needed for Theorem \ref{maxprincipal} when $F$ is of type 1:\\

\noindent $\mathbf{H_\Lambda^1}(F):$ There exists $(a_n,b_n)_{n\ge1}\in{\cal C}(F)$ such that
$(\rho_n)_{n\ge1}$ and $(\beta_n)_{n\ge1}$ converge to some finite values.\\ 

\noindent $\mathbf{H_\Lambda^2}(F):$ There exists a function $g$ satisfying \eqref{lambdaproperty}, a positive number $\lambda$  and $\delta\in(-\infty,x_F)$ such that for sufficiently large $n$,
$$\int_u^0\frac{g(\theta_n)}{g(\theta_n+g(\theta_n)v)}dv\le -\lambda u\quad \textnormal{for every}\;  u\in\left(\frac{\delta-\theta_n}{g(\theta_n)},0\right).$$
 \begin{Remarque} \label{remarkong} 
Assumption $\mathbf{H_\Lambda^1}(F)$ is only needed when $F$ is of type 1 since it is always satisfied in the other cases. Actually, when $F$ is of type $(2,\alpha)$ or $(3,\alpha)$, one can build  $(a_n,b_n)_{n\ge1}\in{\cal C}(F)$ such that the pair of sequences $(\rho_n,\beta_n)_{n\ge1}$ converges to $(-1/\alpha,0)$ and $(1/\alpha,0)$ respectively (see Section \ref{resume} for details). We are not able to obtain such a general result when $F$ is of type 1 but one can check that $\mathbf{H_\Lambda^1}(F)$ is true in some standard cases: \\
\noindent $\bullet$ Exponential distribution: we can take $ a_n=1$ and
$b_n=\log
n$. Then, $\rho_n=0,$ and $\beta_n\xrightarrow{n\rightarrow+\infty} -1.$\\

\noindent $\bullet$ Normal distribution: $a_n=\frac{1}{\sqrt{2\log n}}$ and $b_n=\sqrt{2\log n}-\frac{1}{ 2\sqrt{\log n}}(\log\log n+\log 4\pi)$ belong to ${\cal C}(F)$. For these choices,  $\rho_n\xrightarrow{n\rightarrow+\infty}0$ and $\beta_n\xrightarrow{n\rightarrow+\infty} -1.$ \\
Note that the limits of $(\rho_n)_{n\ge1}$ and $(\beta_n)_{n\ge1}$ correspond to $\rho_G$ and $\beta_G$ defined in Theorem \ref{maxprincipal}. In particular, they depend only 
on the type of $F$ in the sense that if  $(\rho_n)_{n\ge1}$ and $(\beta_n)$ converge, then, the limits are systematically $\rho_G$ and $\beta_G$ respectively (see Section \ref{resume} for details).\\
Assumption  $\mathbf{H_\Lambda^2}(F)$ is needed to control the evolution of $(X_k)_{k\ge1}$ (see Lemma \ref{lemme2}).
First, using that for $g$ and
$\tilde{g}$  satisfying \eqref{lambdaproperty}, $g(t){\sim}\tilde{g}(t)$  as $t\nearrow x_F$, one observes that Assumption $\mathbf{H_\Lambda^2}(F)$ does not depend on the
choice of the function $g$. Second, $\mathbf{H_\Lambda^2}(F)$ is satisfied in the two preceding examples. Actually, for an exponential distribution, one can take $g(x)=1$ and for a normal distribution, one can show that $g(x)\overset{x\rightarrow+\infty}{\sim}x^{-1}$ (see $e.g.$ 
\cite{leadbetter}). Hence, for a sufficiently large $\delta$ and $n$,
$$\frac{g(\theta_n)}{g(\theta_n+g(\theta_n)v)}\le C\frac{g(\theta_n)v+\theta_n}
{\theta_n}\le C\quad \forall v\in\left(\frac{\delta-\theta_n}{g(\theta_n)},0\right).$$  More generally, one notices that $\mathbf{H_\Lambda^2}(F)$ is satisfied  as soon as, $g$ is bounded
and cut by a positive number or,  if there exists
$\delta<x_F$ such that $g$ is noincreasing on $[\delta, x_F)$.
\end{Remarque}

\noindent
Let us now state our main result:
\begin{theorem}\label{maxprincipal} Let  $F$ be a distribution function of type
1, $(2,\alpha)$ with $\alpha>2$ or $(3,\alpha)$ with $\alpha>0$ and let $G$ denote one of the extreme values distribution functions. Assume $\mathbf{H_\Lambda^1}(F)$ and $\mathbf{H_\Lambda^2}(F)$ if
 $F$ is of type 1. Then, for every $(a_n,b_n)_{n\ge1}\in{\cal C}(F)$, the sequence of c\`adl\`ag processes
 $({\cal X}^{(n)})_{n\ge 0}$ (defined in \eqref{process}) converges weakly, for the Skorokhod topology on $\mathbb{D}(\ER_+,\ER)$, to  a stationary Markov process with 
 invariant distribution  $\nu_G$ and infinitesimal generator ${\cal A}$ defined for every $f\in{\cal C}_K^1(D_G)$ by
$${\cal A}f(x)=(\rho_G
x+\beta_G)f'(x)+\int_0^{+\infty}\Big(f(x+y)-f(x)\Big)\varphi_G(x+y)dy $$
where $\varphi_G(x)=-\tau_G'(x)$ and $(\rho_G,\beta_G)$ satisfies
\begin{equation}\label{rhobeta}
(\rho_G,\beta_G)=\begin{cases}(0,-1)&\textnormal{if $F$ is of type 1,}\\
(-1/\alpha,0)&\textnormal{if $F$ is of type $(2,\alpha)$,}\\
(1/\alpha,0)&\textnormal{if $F$ is of type $(3,\alpha)$.}
\end{cases}
\end{equation}
\end{theorem}
\noindent 
With the time discretization standpoint adopted in this paper, Theorem
\ref{maxprincipal} is close to that obtained in  \cite{basak} where the authors show a similar functional weak convergence
result for the Euler scheme with decreasing step of ergodic Brownian diffusions. Then, if $(X_n)$ was a \guil true" Euler scheme
for the Markov process with infinitesimal generator ${\cal A}$, Theorem \ref{maxprincipal} would be only an adaptation
of \cite{basak} to this type of SDE's with jumps.
\begin{Remarque} In the literature about the numerical approximation of the stationary regime of Markovian SDE's, another 
type of result could also be connected to extreme value theory. Actually, in \cite{LP1,LP2}, the authors show the $a.s.$ weak convergence of some weighted occupation measures of the Euler scheme of Brownian SDE's (see also \cite{lemaire} and \cite{panloup} for extensions). By adapting the approach of these papers to this context, it could be possible to retrieve the $a.s.$ CLT for extreme values obtained in \cite{fahrner}.
\end{Remarque}
\begin{Remarque} The reader can observe that we assume $\alpha>2$. This assumption can be viewed as a consequence of martingale methods in which the convergence needs some control of the moments. However, this assumption could be alleviated  and it seems that at the price of technicalities, the result still holds if we only assume that $\alpha>1$.   
\end{Remarque}
\noindent The proof of Theorem \ref{maxprincipal} is divided in three parts. First, in Section
 \ref{section1}, we establish some stability properties for the sequence $(X_k)_{k\ge1}$ and obtain the tightness of the sequence  $({\cal X}^{(n)})_{n\ge 1}$ on
$\mathbb{D}(\ER_+,\ER)$. Second, in Section \ref{identifilimit}, we identify the limit by showing that every weak limit  of $({\cal X}^{(n)})_{n\ge 1}$ is a solution to a martingale problem. Finally, Section \ref{resume} is devoted to some details, to a synthesis of the two preceding parts and to the uniqueness of the martingale problem.\\
Before going further, let us precise several notations of the proof. For a non-decreasing function $g$, we denote by $g^{\leftarrow}$, its left continuous inverse defined by $g^{\leftarrow}(x)=\inf\{y,g(y)\ge x\}$. Throughout  the proof, $C$ will denote a constant which may change from line to line.

\section{Tightness of the sequence $({\cal X}^{(n)})_{n\ge
1}$}\label{section1}

The main result of this section is Proposition \ref{tight} where
we obtain that the sequence $({\cal X}^{(n)})_{n\ge 1}$ is
tight on $\mathbb{D}(\ER_+,\ER)$. Before stating it, we
need to establish a series of technical lemmas.\\
In  Lemma \ref{lemme1}, we show that for every $\delta< x_F$, we
can suppose that $F(x)=0$ for every $x\in(-\infty,\delta)$. This
assumption will be convenient for the sequel of the proof.
\begin{lemme} \label{lemme1} Let $F$ be a distribution function and $\delta$ be a real number
such that $\delta<x_F$. Denote by $F_\delta$ the distribution
function defined by $F_\delta(x)=F(x)1_{\{x\ge\delta\}}$. Let
$({\cal X}^{(n)})$ and $({\cal X}^{(n,\delta)})$ denote
some sequences of c\`adl\`ag processes built as in Equation
\ref{process} and corresponding to $F$ and $F_\delta$ respectively.
Then, for every bounded continuous functional
$H:\mathbb{D}(\ER_+,\ER)\rightarrow\ER$,
\begin{equation}\label{hequ}
\ES[H({\cal X}^{(n)})-H({\cal X}^{(n,\delta)})]\xrightarrow{n\rightarrow+\infty}0.
\end{equation}
\end{lemme}
\begin{proof} Let $(\xi_n)_{n\ge 1}$ and $(\xi_n^\delta)_{n\ge 1}$ denote some sequences of $i.i.d.$ random
variables with distribution function $F$ and $F_\delta$
respectively. Since this lemma only depends on the distribution of
these sequences, we can assume that  $(\xi_n)_{n\ge 1}$ and
$(\xi_n^\delta)_{n\ge 1}$ are built as follows :
$\xi_n=F^{\leftarrow}(U_n)$ and $\xi_n^\delta=F_\delta^{\leftarrow}(U_n)$ where
$(U_n)_{n\ge 1}$ is a sequence of $i.i.d.$ random variables such
that $U_1\sim {\cal U}_{[0,1]}$. In particular, we have
$\xi_n=\xi_n^\delta$ on the event $\{U_n>F(\delta)\}$. It follows
that ${\cal X}^{(n)}={\cal X}^{(n,\delta)}$ on the event
$A_n=\bigcup_{k=1}^n\{U_k>F(\delta)\}$. The sequence $(A_n)_{n\ge
1}$ is non-decreasing and $\PE[\lim_{n\rightarrow+\infty}
A_n]=1$ since $\PE[U_1>F(\delta)]>0$. Assertion \eqref{hequ} easily follows.
\end{proof}
\noindent The aim of the following lemma is to obtain a recursive control of
the conditional moments of order 1 and 2 for the last term of
Equation \eqref{convloischema}. Keeping in mind Lemma
\ref{lemme1}, we show that this control is possible at the price of a potentially restriction of the support of $F$.
\begin{lemme} \label{lemme2} Let $F$ be a distribution function  of type 1, $(2,\alpha)$ with $\alpha>2$ or
$(3,\alpha)$ with $\alpha>0$ and let $(a_n,b_n)_{n\ge1}\in{\cal C}(F)$. Assume furthermore that Assumption
$\mathbf{H_\Lambda^2}(F)$ holds (if  $F$ is of type 1). Then,
for every  $\varepsilon>0$ and, $\varepsilon<\alpha-2$ if $F$ is
of type 2, there exists $\delta_\varepsilon<x_F$ such that  for
every $\delta\in[\delta_\varepsilon,x_F[$, for every sequence
$(\xi_k)_{k\ge1}$ of $i.i.d.$ random variables with distribution
function $F_\delta$ (defined by
$F_\delta(x)=F(x)1_{\{x\ge\delta\}}$), there exists
$n_\epsilon\in\EN$ such that for every  $k\ge n_\epsilon$,\\

\noindent$(i)$
\begin{eqnarray*}
\ES\big[\big(\frac{\xi_{k+1}}{a_k}-{X}_{k}-\frac{b_k}{a_k}\big)_+/{\cal
F}_k\big]\le C\gamma_{k+1}\begin{cases}
\big(1+e^{-\lambda{X}_{k}}\big)\quad
(\lambda>0)&\textnormal{ if $F$ is of type 1,} \\
\big(1+|{X}_{k}|^{-(\alpha+\varepsilon)+1}\big)&\textnormal{ if $F$ is of type $(2,\alpha)$,} \\
\big(1+|{X}_{k}|^{\alpha+\varepsilon+1}\big)&\textnormal{ if
$F$ is of type $(3,\alpha)$}.
 \end{cases}
 \end{eqnarray*}
$(ii)$
\begin{eqnarray*}
 \ES\big[\big(\frac{\xi_{k+1}}{a_k}-{X}_{k}-\frac{b_k}{a_k}\big)^2_+/{\cal
F}_k\big]\le C\gamma_{k+1}
\begin{cases}
\big(1+e^{-\lambda{X}_{k}}\big)\quad
(\lambda>0)&\textnormal{ if $F$ is of type 1,} \\
\big(1+|{X}_{k}|^{-(\alpha+\varepsilon)+2}\big)&\textnormal{ if $F$ is of type $(2,\alpha)$,} \\
\big(1+|{X}_{k}|^{\alpha+\varepsilon+2}\big)&\textnormal{ if
$F$ is of type $(3,\alpha)$}.
 \end{cases}
 \end{eqnarray*}
\end{lemme}
\begin{proof} $(i)$ Using \eqref{choixanbn}, it is easy to check that
$(i)$ and $(ii)$ hold for every sequence
$(a_n,b_n)_{n\ge1}\in{\cal C}(F)$  if and only if they hold for
a particular sequence $(a_n,b_n)_{n\ge1}\in{\cal C}(F)$. Hence, we only prove this
lemma with $(a_n,b_n)$ chosen as in Proposition \ref{rappelgne}.\\
We have:
\begin{equation}\label{mom1}
\ES\big[\big(\frac{\xi_{k+1}}{a_k}-{X}_k-\frac{b_k}{a_k}\big)_+/{\cal
F}_k\big]=\int_0^{+\infty}\big(1-F(a_k(u+{X}_k)+b_k)\big)du.
\end{equation}
\noindent Assume first that $F$ is  of type 1. We set
$(a_n,b_n)=(g(\theta_n), \theta_n)$ where $g$ is such that
\eqref{lambdaproperty} holds. By Proposition 1.4 p.43 of
\cite{resnick}, $g$ can be chosen such that the following
representation holds for every $x\in[z_0,x_F)$:
\begin{equation}\label{vonmises}
1-F(x)=c(x)\exp\big(-\int_{z_0}^{x}\frac{1}{g(s)}ds\big)\qquad
\textnormal{where}\quad c(x)\xrightarrow{x\nearrow x_F}c>0.
\end{equation}
Let $\delta$ be a real number such that $\delta\ge z_0$ and
$(\xi_n)$ be a sequence of $i.i.d.$  random variables with
distribution function $F_\delta$. For every positive $u$, we have
$a_k(u+{X}_k)+b_k> \delta$. Then, if
$a_k(u+{X}_k)+b_k<x_F$, we have for sufficiently large $k$,
\begin{align*}
\frac{1-F(a_k(u+{X}_k)+b_k)}{\gamma_k}&=\frac{1-F(a_k(u+{X}_k)+b_k)}{1-F(b_k)}\\
&=\frac{c(a_k(u+{X}_k)+b_k)}{c(b_k)}\exp\Big(-\int_{b_k}^{b_k+a_k(u+{X}_k)}\frac{1}{g(s)}ds\Big)\\
&=\frac{c(a_k(u+{X}_k)+b_k)}{c(b_k)}\exp\Big(-\int_0^{u+{X}_k}\frac{a_k}{g(a_k
s+b_k)}ds\Big).
\end{align*}
First, there exists $\delta_1\in[z_0,x_F[$ such that for every
$x\in[\delta_1, x_F[$, $c/2\le c(x)\le 2c$. Hence, for every
$\delta\ge\delta_1$ and for sufficiently large $k$,
$$\frac{c(a_k(u+{X}_k)+b_k)}{c(b_k)}\le C.$$
Second, one derives from Assumption $\mathbf{H_\Lambda^2}(F)$, one can find 
$\delta_2\in[z_0,x_F[$ such that for sufficiently large $k$,
$$\int_{u+{X}_k}^0\frac{a_k}{g(a_k
s+b_k)}ds\le -\lambda (u+{X}_k)\quad\textnormal{when
$u+{X}_k<0$}. $$
Thus, for $\delta>\delta_0=\delta_1\vee\delta_2$ and sufficiently large $k$,
\begin{equation}\label{trp1}
\frac{1-F(a_k(u+{X}_k)+b_k)}{\gamma_k}\le
C\exp(-\lambda(u+{X}_k)) \quad \forall u\in[0,-{X}_k\vee0).
\end{equation}  
Finally, if $u+{X}_k>0$, we derive from Lemma 2.2 p.78
of \cite{resnick} that, for every $\varepsilon>0$, there exists
$k_\varepsilon>0$ such that
$$\frac{a_k}{g(a_k s+b_k)}\ge \frac{1}{1+\varepsilon s}\qquad\forall
k\ge k_\varepsilon\quad\forall s>0.$$ Hence, for sufficiently large $k$,
\begin{equation}\label{trp2}
\frac{1-F(a_k(u+{X}_k)+b_k)}{\gamma_k}\le C_\varepsilon
\big(1+\varepsilon(u+{X}_k))^{-\frac{1}{\varepsilon}}\quad\forall u\in[0\vee -X_k,+\infty) .
\end{equation}
 Setting
$\varepsilon=1/2$ and $x^-=\max(-x,0)$, we derive from \eqref{mom1}, \eqref{trp1} and \eqref{trp2}
that, with a sufficient restriction of the support of $F$, there exists $k_0\in\EN$ such that for every $k\ge k_0$,
\begin{align*}
\ES\big[\big(\frac{\xi_{k+1}}{a_k}-{X}_k-\frac{b_k}{a_k}\big)_+/{\cal
F}_k\big]&\le
C\gamma_k\left(\int_0^{{X}_k^-}e^{-\lambda(u+{X}_k)}du
+\int_{{X}_k^-}^{+\infty}
\frac{1}{(1+\frac{1}{2}(u+{X}_k))^2}du\right)\\
&\le C\gamma_k\big(1+e^{-\lambda{X}_k}\big).
\end{align*}
Assume  now that $F$ is of type $(2,\alpha)$. By the  Karamata
representation (\cite{resnick}, p. 58), for every $x\ge 1$, we
have
\begin{equation}\label{karama}
1-F(x)=c(x) \exp\big(-\int_1^x\frac{\alpha(t)}{t}dt\big)
\end{equation}
with $c(x)\xrightarrow{x\rightarrow+\infty}c>0$ and
$\alpha(x)\xrightarrow{x\rightarrow+\infty}\alpha>0$. Then,
replacing $F$ by $F_1$ defined by $F_1(x)=F(x)1_{\{x\ge1\}}$, we
have  for every $u>0$ and sufficiently large $k$,
\begin{align*}
\frac{1-F(a_k(u+{X}_k)+b_k)}{\gamma_k}&=\frac{1-F(a_k(u+{X}_k))}{1-F(a_k)}\\
&=\frac{c(a_k(u+{X}_k))}{c(a_k)}\exp\left(-\int_{a_k}^{a_k(u+{X}_k)}\frac{\alpha(s)}{s}ds\right).
\end{align*}
Let $\varepsilon>0$. Let $\delta>1$ such that  $c/2\le c(x)\le 2c$
and $\alpha-\varepsilon\le \alpha(x)\le \alpha+\varepsilon$ for
every  $x\ge\delta$. Then, replacing $F_1$ by $F_\delta$, we
obtain that
\begin{equation}\label{eq1}
\frac{1-F(a_k(u+{X}_k)+b_k)}{\gamma_k}\le
C\Big((u+{X}_k)^{-\alpha+\varepsilon}1_{\{u+{X}_k>1\}}+
(u+{X}_k)^{-\alpha-\varepsilon}1_{\{u+{X}_k\le1\}}\Big).
\end{equation}
Hence, we derive from \eqref{mom1} that for
$\varepsilon\in]0,\alpha-1[$ and sufficiently large
$k$,
\begin{equation}
\ES\big[\big(\frac{\xi_{k+1}}{a_k}-{X}_k-\frac{b_k}{a_k}\big)_+/{\cal
F}_k\big]\le
C\gamma_k\Big(1+\frac{1}{{X}_k^{\alpha+\varepsilon-1}}\Big).
\end{equation}
Finally, assume  that $F$ is of type $(3,\alpha)$. By Corollary
1.14 p.62 of \cite{resnick}, $1-F$ can be written as follows:
\begin{equation}\label{karama2}
1-F(x)=c(x)\exp\big(-\int_{x_F-1}^x\frac{\alpha(s)}{x_F-s}ds\big)\qquad \forall x<x_F
\end{equation}
with $c(x)\xrightarrow{x\nearrow x_F}c>0$ and
$\alpha(x)\xrightarrow{x\nearrow x_F}\alpha>0$. Let
$\varepsilon\in]0,\alpha/2[$. There exists
$\delta_\varepsilon<x_F$ such that for every
$x>\delta_\varepsilon$, $c/2\le c(x)\le 2c$ and
$\alpha-\varepsilon\le \alpha(x)\le \alpha+\varepsilon.$ Assume
now that $(\xi_n)_{n\ge1}$ is a sequence of $i.i.d.$ random
variables with distribution function $F_\delta$. Setting
$(a_k,b_k)=(x_F-\theta_k,x_F)$, we obtain that for every
$\varepsilon>0$, there exists $n_\varepsilon\in\EN$ such that for
every $k\ge n_\varepsilon$ and $u\in[0,-{X}_k[$,
\begin{align*}
\frac{1-F(a_k(u+{X}_k)+b_k)}{\gamma_k}&=\frac{1-F(a_k(u+{X}_k)+x_F)}{1-F(x_F-a_k)}\\
&\le\frac{c(a_k(u+{X}_k)+x_F)}{c(x_F-a_k)}\exp\left(-\int_{x_F-a_k}^{x_F+a_k(u+{X}_k)}
\frac{\alpha(s)}{x_F-s}ds\right),\\&\le
C\exp\Big(\int_{a_k}^{-a_k(u+{X}_k)}\frac{\alpha(x_F-v)}{v}dv\Big),\\
\le C\Big(&(-(u+{X}_k))^{\alpha-\varepsilon}
1_{\{u+{X}_k>-1\}}+(-(u+{X}_k))^{\alpha+\varepsilon}1_{\{u+{X}_k\le-1\}}\Big).
\end{align*}
Using that $1-F(a_k(u+{X}_k)+b_k)=0$ when $u+{X}_k>0$, we
finally derive from \eqref{mom1} that
$$\ES\big[\big(\frac{\xi_{k+1}}{a_k}-{X}_k-\frac{b_k}{a_k}\big)_+/{\cal
F}_k\big]\le
C\gamma_k(1+|{X}_k|^{1+\alpha+\varepsilon}).$$ \\
\\
$(ii)$ We have
\begin{equation}\label{mom2}
\ES\left[\big(\frac{\xi_{k+1}}{a_k}-{X}_k-\frac{b_k}{a_k}\big)^2_+/{\cal
F}_k\right]=2\int_0^{+\infty}u\left(1-F(a_k(u+{X}_k)+b_k)\right)du.
\end{equation} and the result follows easily from the controls established
in $(i)$. Details are left the reader.
\end{proof}
\noindent In the next lemma, we obtain an uniform control of the
moments of the sequence $({X}_n)_{n\ge1}$ in terms of the type
of the distribution function $F$.
\begin{lemme}\label{lemme3} Let $F$ be a distribution function
and assume that there exists $\delta<x_F$ such that $F(x)=0$ for every $x<\delta$. Let $(a_n,b_n)\in{\cal C}(F)$. Then,\\
$(i)$
\begin{eqnarray*} \sup_{n\ge
1}\ES[|{X}_n|^r]<+\infty\quad
\begin{cases}\forall r\ge0 &\textnormal{if $F$ is of type 1 or $(3,\alpha)$} \\
\forall r\in[0,\alpha) &\textnormal{if $F$ is of type
$(2,\alpha)$}.
 \end{cases}
 \end{eqnarray*}
\noindent $(ii)$  Assume that $F$ is of  type 1. Then, for every
positive number $\lambda$, there exists $\delta_\lambda<x_F$ such that for every
sequence $(\xi_n)_{n\ge1}$ with distribution function $F_{\delta}$
with $ \delta\in[\delta_\lambda,x_F)$,
\begin{equation}\label{exptype1}
 \sup_{n\ge 1}\ES[\exp(-\lambda{X}_{n})]<+\infty.
 \end{equation}
\noindent $(iii)$ Assume that $F$ is of type $(2,\alpha)$
$(\alpha>0)$ such that $F(x)=0$ when $x<1$. Then,$$\sup_{n\ge
1}\ES\left[\frac{1}{({X}_n)^r}\right]<+\infty, \qquad \forall r\ge 0.$$
\end{lemme}
\begin{proof}
As in Lemma \ref{lemme2}, one derives easily from \eqref{choixanbn} that 
the assertions of Lemma \ref{lemme3} hold
with every $(a_n,b_n)_{n\ge1}\in{\cal C}(F)$ if and only if they
hold for a particular
sequence $(a_n,b_n)_{n\ge1}\in{\cal C}(F)$. Hence, we assume that $(a_n,b_n)_{n\ge1}$
is as specified in  Proposition \ref{rappelgne}.\\
$(i)$ This is a consequence of  Proposition 2.1
p.77  of \cite{resnick}. \\
$(ii)$ Let $\delta\in\ER$ with $\delta<x_F$ and let
$(\xi_n)_{n\ge1}$ be a sequence of $i.i.d.$ random variables with
distribution function $F_\delta$. First, since $F_\delta$ is of
type 1 for every $\delta<x_F$, $({X}_{n})_{n\ge 1}$ converges
in distribution. Second, $x\mapsto e^{-\lambda x}$ is bounded on
$[-L,+\infty)$ for every $L>0$. Hence, \eqref{exptype1} holds if
there exists $\delta_\lambda<x_F$ such that for every
$\delta\in[\delta_\lambda,x_F[$,
\begin{equation}\label{equiint}
\lim_{L\rightarrow+\infty}\limsup_{n\rightarrow+\infty}\ES\left[\exp(-\lambda {X}_n)1_{\{{X}_n<-L\}}\right]=0.
\end{equation}
Let us prove \eqref{equiint}. We have:
\begin{align*}
\ES\left[\exp(-\lambda{X}_{n})1_{\{{X}_{n}<-L\}}\right]&\le\ES\left[\exp\left(-\lambda{X}_{n}1_{\{{X}_{n}<-L\}}\right)\right],\\&
\le 1+\int_0^{+\infty}\lambda\exp(\lambda u)\PE({X}_{n}1_{\{X_n<-L\}}<-u)du,\\
&\le 1+C\PE({X}_{n}<-L)+\int_L^{+\infty}\lambda\exp(\lambda u)\PE({X}_{n}<-u)du.
\end{align*}
Since $\xi_1$ has distribution function $F_\delta$,
$\PE({X}_n<-u)=0$ if $u>a_n^{-1}(b_n-\delta)$ and  $\PE({X}_n<-u)=F^n(-a_n
u+b_n)$ if $u<a_n^{-1}(b_n-\delta)$. Hence, we derive from Lemma
2.2 p.78 of \cite{resnick} that,  for every $\varepsilon>0$, there
exists $\delta_\varepsilon<x_F$ such that for every
$\delta\in[\delta_\epsilon,x_F)$, for every
$u>a_n^{-1}(\delta-b_n)$,
$$\PE({X}_n<-u)=F^n(-a_n u+b_n)\le
\exp\Big(-(1-\varepsilon)^2(1+\varepsilon|u|)^{\frac{1}{\varepsilon}}\Big).$$
Setting $\varepsilon:=1/2$, we obtain that for such $\delta$,
$u\mapsto \exp(\lambda u)\PE({X}_{n}<-u)$ is dominated on
$\ER_-$ by a Lebesgue-integrable function( uniformly in $n$). Hence, we derive from
the dominated convergence theorem that
$$\lim_{L\rightarrow+\infty}\limsup_{n\rightarrow+\infty}
\int_L^{+\infty}\lambda \exp(\lambda u)\PE({X}_{n}<-u)du=0.$$
Finally, since
$$\lim_{L\rightarrow+\infty}\limsup_{n\rightarrow+\infty}\PE({X}_{n}<-L)=
\lim_{L\rightarrow+\infty}\PE(X_\infty<-L)=0,$$
\eqref{equiint} follows.\\
$(iii)$ Assume that $F(x)=0$ for every $x<1$. In particular,
${X}_n>0$ $a.s.$ for every $n\ge 1$. Then,
$$\frac{1}{{X}_n}=\frac{a_n}{\max(\xi_1,\ldots,\xi_n)}=-a_n
\max(\tilde{\xi_1},\ldots,\tilde{\xi_n})$$ where
$\tilde{\xi_1}=-\frac{1}{\xi_1}$.  Since $\xi_1$ is of type
$(2,\alpha)$, it is easy to check that $\tilde{\xi_1}$ is of type
$(3,\alpha)$ with $x_F=0$. Set $\tilde{F}(x):=\PE(\tilde{\xi_1}\le
x)$ and $\tilde {a}_n:=-\inf\{x, \tilde{F}(x)\ge 1-\frac{1}{n}\}$.
We have
$$-\tilde{a}_n=\inf\{x, F(-\frac{1}{x})\ge
1-\frac{1}{n}\}=-\frac{1}{a_n},$$ and the result follows from
$(i)$ (Control of the moments for $F$ of type $(3,\alpha)$).
\end{proof}

\noindent In Lemma \ref{lemmetension}, we state a simple criteria of
$C$-tightness adapted to this problem.

\begin{lemme} \label{lemmetension}Let $(Z^{(n)})_{n\ge 1}$ be a sequence of
c\`adl\`ag processes such that for every $n\ge1$,
\begin{equation}\label{defZN}
Z_t^{(n)}=\sum_{k=n+1}^{N(n,t)}\gamma_k\phi_k({X}_{k-1})
\end{equation}
where $(\phi_k)_{k\ge1}$ is a sequence of real functions and
assume that there exists $p>1$ such that $\sup_{k\ge
1}\ES[|\phi_k({X}_{k-1})|^p]<+\infty$. Then,
$(Z^{(n)})_{n\ge 1}$ is  $C$-tight.
\end{lemme}
\begin{proof} First, by \eqref{defZN} and the fact that $\sup_{k\ge0}\ES[|\phi_k(X_{k-1})|]<+\infty$, we have for every positive $T$ and $K$, 
\begin{align*}
\PE(\sup_{t\in[0,T]}|Z_t^{(n)}|>K)&\le\frac{1}{K}\ES[\sup_{t\in[0,T]}|Z_t^{(n)}|]\le\frac{1}{K}\sum_{k=n+1}^{N(n,T)}\gamma_k\sup_{k\ge0}\ES[|\phi_k(X_{k-1})|]\\
&\le \frac{C (\Gamma_{N(n,T)}-\Gamma_n)}{K}\le \frac{CT}{K}
\end{align*}
thanks to the definition of $N(n,T)$. Hence, for every positive $T$,
$$\PE(\sup_{t\in[0,T]}|Z_t^{(n)}|>K)\xrightarrow{K\rightarrow+\infty}0.$$
Then, according to Theorem VI.3.26 of \cite{jacodshiryaev}, we have to show that for every positive $T$, $\varepsilon$ and $\eta$, there exists $\delta>0$ and $n_0\in\mathbb{N}$ such that
$$\PE(\sup_{|t-s|\le\delta,0\le s\le t\le T}|Z_t^{(n)}-Z_s^{(n)}|\ge\varepsilon)
\le\eta\quad\forall n\ge n_0.$$
In fact, using for instance proof of Theorem 8.3 of \cite{billingsley}, it suffices to show that for every
positive  $\varepsilon,\eta$ and  $T$, there exists $\delta>0$ and
$n_0\ge 1$ such that:
\begin{equation}\label{ty}
\frac{1}{\delta}\mathbb{P}(\sup_{t\le s\le
t+\delta}|Z^{(n)}_t-Z^{(n)}_s|\ge\varepsilon)\le\eta\qquad\forall
n\ge n_0\quad\textnormal{and}\quad 0\le  t\le T.
\end{equation}
By \eqref{defZN}, for every $t\in [0,T]$ and $s\in[t,t+\delta[$,
$$|Z^{(n)}_t-Z^{(n)}_s|\le
C\sum_{k=N(n,t)+1}^{N(n,s)}\gamma_k|\phi_k({X}_{k-1})|\le
C\sum_{k=N(n,t)+1}^{N(n,t+\delta)}\gamma_k|\phi_k({X}_{k-1})|.$$
Let $p>1$ such that $\sup_{k\ge
1}\ES[|\phi_k({X}_{k-1})|^p]<+\infty$ and set
$\alpha_k=\gamma_k^{1-\frac{1}{p}}$ and
$\mu_k=\gamma_k^{\frac{1}{p}}|\phi_k({X}_{k-1})|$. We derive
from the Holder inequality with $\bar{p}=\frac{p}{p-1}$ and
$\bar{q}=p$ that
$$\sum_{k={N}(n,t)+1}^{{N}(n,t+\delta)}\gamma_k|\phi_k({X}_{k-1})|\le
\left(\sum_{k={N}(n,t)+1}^{{N}(n,t+\delta)}\gamma_k\right)^{\frac{p-1}{p}}
\left(\sum_{k={N}(n,t)+1}^{{N}(n,t+\delta)}\gamma_k|\phi_k({X}_{k-1})|^p\right)^\frac{1}{p}.$$
It follows from the Markov inequality that
$$\mathbb{P}(\sup_{t\le s\le
t+\delta}|Z^{(n)}_t-Z^{(n)}_s|\ge\varepsilon)\le\frac{1}{\varepsilon^p}\left(\sum_{k={N}(n,t)+1}^{{N}(n,t+\delta)}\gamma_k\right)^{\frac{p-1}{p}+1}
\sup_{n\ge 1}\ES[|\phi_n({X}_{n-1})|^p].$$ Since $\sup_{n\ge
1}\ES[|\phi_n({X}_{n-1})|^p]<+\infty$,
$\sum_{k={N}(n,t)+1}^{{N}(n,t+\delta)}\gamma_k\le 2\delta$ for
sufficiently large $n$, we deduce that \eqref{ty} holds for
sufficiently small $\delta$.
\end{proof}
\noindent We can now state the main result of this section:
\begin{prop}\label{tight}
Let $F$ be a distribution function of  type 1, $(2,\alpha)$ with $\alpha>2$ or, $(3,\alpha)$
with $\alpha>0$, and assume $\mathbf{H_\Lambda^2}(F)$ if $F$ is of type 1. 
Let $(a_n,b_n)\in{\cal C}(F)$ such that $(\rho_n)_{n\ge1}$ and $(\beta_n)_{n\ge1}$ are bounded. 
Then,  the sequence $({\cal X}^{(n)})_{n\ge 1}$ is
tight on $\mathbb{D}(\ER_+,\ER)$.
\end{prop}
\begin{proof} Using the convention
$\underset{\emptyset}{\sum}=0$,  ${\cal X}_t^{(n)}$ can be
written as follows:
\begin{align*}
&{\cal X}_t^{(n)}=X_n+D_t^{(n)}+Y_t^{(n)}\quad \textnormal{where,}\\ D_t^{(n)}=&\sum_{k=n+1}^{N(n,t)}\gamma_k\Big(\rho_k{X}_{k-1}+\beta_k
+(1+\rho_k\gamma_k)h_k({X}_{k-1})\Big),\\
Y_t^{(n)}=&\sum_{k=n+1}^{N(n,t)}(1+\rho_k\gamma_k)\Delta\bar{Y}_{k}\quad\textnormal{with}\quad
h_k(x)=\frac{1}{\gamma_k}\ES\left[\big(\frac{\xi_1}{a_{k-1}}-x-\frac{b_{k-1}}{a_{k-1}}\big)_+\right]\quad\textnormal{and,}\\
\Delta\bar{Y}_{k}=&\big(\frac{\xi_{k}}{a_{k-1}}-{X}_{k-1}-\frac{b_{k-1}}{a_{k-1}}\big)_+-
\gamma_k h_k({X}_{k-1}).
\end{align*}
By Lemma \ref{lemme1}, we can assume that there exists
$\delta<x_F$ such that $F(x)=0$ for every $x<\delta$. Then, in
this case, it follows from Lemma \ref{lemme3} and from the
assumptions on $\alpha$ when $F$ is of type $(2,\alpha)$ that
$(Y_t^{(n)})_{n\ge 1}$ is a square-integrable martingale. We have
\begin{align*}
\psg Y^{(n)}\psd_t&=\sum_{k=n+1}^{N(n,t)}(1+\rho_k\gamma_k)^2\ES\left[(\Delta\bar{Y}_{k})^2/{\cal
F}_{k-1}\right]\\
&=\sum_{k=n+1}^{N(n,t)}(1+\rho_k\gamma_k)^2\left(
\ES\left[\big(\frac{\xi_{k}}{a_{k-1}}-{X}_{k-1}-\frac{b_{k-1}}{a_{k-1}}\big)^2_+/{\cal
F}_{k-1}\right]-\gamma_k^2 h_k^2({X}_{k-1})\right).
\end{align*}
Then, it is standard that if we want to show that  $({\cal X}^{(n)})_{n\ge 1}$ is
tight on $\mathbb{D}(\ER_+,\ER)$ it suffices to prove that  $(D^{(n)})_{n\ge 1}$ and
$(\psg Y^{(n)}\psd)_{n\ge 1}$ are $C$-tight (see $e.g.$
\cite{jacodshiryaev}). Let us show this last assertion. First, we
observe that
\begin{align*}
&D_t^{(n)}=\sum_{k=n+1}^{N(n,t)}\gamma_k\phi_{k,1}({X}_{k-1})
\quad\textnormal{and}\quad
\psg Y^{(n)}\psd_t=\sum_{k=n+1}^{N(n,t)}\gamma_k\phi_{k,2}({X}_{k-1})\quad\textnormal{with,}\\
&\phi_{k,1}(x)=\rho_k x+\beta_k
+(1+\rho_k\gamma_k)h_k(x)\qquad \textnormal{and,}\\
&\phi_{k,2}(x)=\frac{(1+\rho_k\gamma_k)^2}{\gamma_k}\left(\ES\left[\big(\frac{\xi_{1}}{a_{k-1}}-x-\frac{b_{k-1}}{a_{k-1}}\big)^2_+\right]-\gamma_k^2
h_k^2(x)\right).
\end{align*}
Let $\varepsilon>0$ and $\delta_\varepsilon<x_F$ such that Lemma
\ref{lemme2} holds and assume that $F(x)=0$ for every
$x<\delta_\varepsilon$. Since $(\rho_k)$ and $(\beta_k)$
are bounded, we derive from Lemma \ref{lemme2} and \ref{lemme3}
that
\begin{eqnarray*}
\phi_{k,1}(X_{k-1})\le C(1+|X_{k-1}|)+C\begin{cases}e^{-\lambda X_{k-1}}\quad
(\lambda>0)&\textnormal{ if $F$ is of type 1,} \\
|X_{k-1}|^{-(\alpha+\varepsilon)+1}&\textnormal{ if $F$ is of type $(2,\alpha),$} \\
|X_{k-1}|^{\alpha+\varepsilon+1}&\textnormal{ if $F$ is of type
$(3,\alpha)$,}
 \end{cases}
 \end{eqnarray*}
and
\begin{eqnarray*}
\phi_{k,2}(X_{k-1})\le
C\begin{cases}\big(1+e^{-\lambda X_{k-1}}\big)\quad
(\lambda>0)&\textnormal{ if $F$ is of type 1,} \\
\big(1+|X_{k-1}|^{-(\alpha+\varepsilon)+2}\big)&\textnormal{ if $F$ is of type $(2,\alpha)$,} \\
\big(1+|X_{k-1}|^{\alpha+\varepsilon+2}\big)&\textnormal{ if $F$ is of
type $(3,\alpha)$.}
 \end{cases}
 \end{eqnarray*}
Then, we derive from Lemma \ref{lemme3} that $\sup_{k\ge
1}\ES[|\phi_{k,i}({X}_{k-1})|^2]<+\infty$ for $i=1,2$ and it
follows from Lemma \ref{lemmetension} that $(D^{(n)})_{n\ge 1}$
and $(\psg Y^{(n)}\psd)_{n\ge 1}$ are $C$-tight.
\end{proof}
\section{Characterization of the limit}\label{identifilimit}
Denote by ${\cal A}_\rho^\beta$ the operator defined on ${\cal C}^1_K(D_G)$ by
\begin{equation}\label{arhobeta}
{\cal A}_\rho^\beta f(x)=(\rho x+\beta)f'(x)+\int_0^{+\infty}\Big(f(x+y)-f(x)\Big)\varphi_G(x+y)dy. 
\end{equation}
The main objective  of this section is to show that every weak limit ${\cal X}^{\infty}$
of $({\cal X}^{(n)})_{n\ge1}$ solves the martingale problem $({\cal A}_\rho^\beta,\nu_G, {\cal C}^1_K(D_G))$, $i.e.$
we want to prove that if ${\cal X}^{\infty}$ exists, the two following properties are satisfied: ${\cal L}({\cal X}^{\infty}_0)=\nu_G$ and for every $f\in{\cal C}^1_K(D_G)$, $(M_t^f)$ defined by
$$M_t^f=f({\cal X}^{\infty}_t)-f({\cal X}^\infty_0)-\int_0^t A_\rho^\beta f({\cal X}^\infty_s)ds$$
is a martingale. These properties are obtained in Proposition \ref{propvalda}. The main tool for this result is Lemma \ref{lemme5} where we show that the increments of $X_n$ are \guil asymptotically equal" to those of the limit process. In this way, we need to establish the following lemma.
 \begin{lemme}\label{lemmeuniform} Let $F$ be of type 1, $(2,\alpha)$ ou $(3,\alpha)$ and
let $(a_n,b_n)\in{\cal C}(F)$.
Then,
\begin{equation*}
\sup_{x\in K}\Big(\frac{1-F(a_n
x+b_n)}{1-F(\theta_n)}-\tau_G(x)\Big)\xrightarrow{n\rightarrow+\infty}0\quad
\textit{for any compact subset $K$ of $D_G$.}
 \end{equation*}
\end{lemme}
\begin{proof} Assume first that $F$ is of type 1. Using that $a_k\sim g(\theta_k)$ and 
$b_k-\theta_k=o(g(\theta_k))$ when $k\longrightarrow+\infty$, we derive from the Von Mises Representation (see \eqref{vonmises}) that for sufficiently large $k$,
\begin{equation}
\frac{1-F(a_k x+b_k)}{1-F(\theta_k)} =\frac{c(a_k
x+b_k)}{c(\theta_k)}\exp\left(-\int_0^{x+\varepsilon_1(k,x)}\frac{g(\theta_k)}{g(\theta_k+sg(\theta_k))}ds\right)
\end{equation}
where $c(x)\xrightarrow{x\nearrow x_F}c>0$ and $(\varepsilon_1(k,.))_{k\ge1}$ is a sequence of functions
which converges locally uniformly on $\ER$ to 0. According to Lemmas
1.2 p.40 and 1.3 p.41 of \cite{resnick}, the sequences of functions $(x\mapsto a_k x+b_k)_{k\ge1}$
and $(x\mapsto\frac{g(\theta_k)}{g(\theta_k+x g(\theta_k))})_{k\ge1}$ converge locally
uniformly on $\ER$, to $x_F$ and $1$ respectively. The result follows easily in this case.\\

\noindent Suppose now that $F$ is of type $(2,\alpha)$ and consider a  compact
subset $K$ of $(0,+\infty)$. Using that $a_k\sim\theta_k$ and $b_k=o(\theta_k)$ when $k\rightarrow+\infty$, the Karamata Representation (see \eqref{karama}) yields for sufficiently large $k$,
\begin{equation*}
\frac{1-F(a_kx+b_k)}{1-F(\theta_k)}=\frac{c(a_kx+b_k)}{c(\theta_k)}\exp\left(-\int_{1}^{x+\varepsilon_2(k,x)}\frac{\alpha(\theta_k
s)}{s}ds\right).
\end{equation*}
where $c(x)\xrightarrow{x\rightarrow+\infty}c>0$,
$\alpha(x)\xrightarrow{x\rightarrow+\infty}\alpha>0$ and $(\varepsilon_2(k,.))_{k\ge1}$ is a sequence of functions
which converges uniformly on $K$ to 0. The result follows in this case from the fact that 
$\theta_n\xrightarrow{n\rightarrow+\infty}+\infty$.\\
Finally, if $F$ is of type $(3,\alpha)$, one considers a compact subset $K$
of $\ER_-$. Using  that $a_k\sim x_F-\theta_k$ and $b_k-x_F=o(a_k)$, one derives from Corollary 1.14 p.62 of \cite{resnick} that for sufficiently large $k$ and for every $x\in K$,
\begin{align*}
&\frac{1-F(a_kx+b_k)}{1-F(\theta_k)}=\frac{c(a_kx+b_k)}{c(\theta_k)}
\exp\left(-\int_{-1}^{x+\varepsilon_3(k,x)}\frac{\alpha((x_F-\theta_k)v+x_F)}{v}dv\right).
\end{align*}
where $c(x)\xrightarrow{x\nearrow x_F}c>0$,
$\alpha(x)\xrightarrow{x\nearrow x_F}\alpha>0$ and $(\varepsilon_2(k,.))_{k\ge1}$ is a sequence of functions
which converges uniformly on $K$ to 0. Therefore, the result follows from the fact that
$x_F-\theta_n\xrightarrow{n\rightarrow+\infty}0$.
\end{proof}
\noindent 
\begin{lemme}\label{lemme5}
Let $F$ be of type 1, $(2,\alpha)$  with $\alpha>2$ or $(3,\alpha)$ and let $(a_n,b_n)_{n\ge1}\in{\cal C}(F)$ such that $(\rho_n)$ and $(\beta_n)$ converge to some finite numbers $\rho$ and $\beta$ respectively and such that the assertions of Lemmas \ref{lemme2} and \ref{lemme3} hold. Then,  for every $f\in {\cal C}_K^1(D_G)$,
$$\ES[f({ X}_{n+1})-f({ X}_n)/{\cal
F}_{n}]=\gamma_{n+1}{\cal A}_\rho^\beta f({ X}_n)+\gamma_{n+1}R_n$$  where
${\cal A}_\rho^\beta$ is defined by \eqref{arhobeta} and $(R_n)$ is an $({\cal F}_n)$-adapted sequence such that $R_n\longrightarrow0$ in $L^1$.
\end{lemme}
\begin{proof}  First, we show that $(R_n)$ defined for every $n\in\EN$ by
$$R_n:=\frac{1}{\gamma_{n+1}}\ES[f({ X}_{n+1})-f({ X}_n)/{\cal
F}_{n}]-{\cal A}_\rho^\beta f({ X}_n)$$
is uniformly integrable.  On the one hand, using that $f$ is a Lipschitz function, it follows from Lemma \ref{lemme2}  that 
\begin{align*}
V_n:=\frac{1}{\gamma_{n+1}}&\Big|\ES[f({ X}_{n+1})-f({ X}_n)/{\cal
F}_{n}]\Big|\\
&\le C(|\rho_{n+1}| |X_n|+|\beta_{n+1}|)+\begin{cases}
\big(1+e^{-\lambda{X}_{n}}\big)\quad
(\lambda>0)&\textnormal{ if $F$ is of type 1,} \\
\big(1+|{X}_{n}|^{-(\alpha+\varepsilon)+1}\big)&\textnormal{ if $F$ is of type $(2,\alpha)$,} \\
\big(1+|{X}_{n}|^{\alpha+\varepsilon+1}\big)&\textnormal{ if
$F$ is of type $(3,\alpha)$}.
\end{cases}
\end{align*}
Therefore, since $(\rho_n)$ and $(\beta_n)$ are bounded, we derive from Lemma \ref{lemme3} that there exists $\eta>1$ such that $$\sup_{n\ge1}\ES[|V_n|^\eta]<+\infty.$$
On the other hand,  one checks that
$A_\rho^{\beta}f$ is a bounded function. Actually, by an integration by parts, $A_\rho^\beta f$ can be written:
\begin{equation}\label{rep1A}
{\cal A}_\rho^\beta f(x)=(\rho x+\beta)f'(x)+\int_0^{+\infty}f'(x+y)\tau_G(x+y)dy. 
\end{equation}
Therefore, since $f\in{\cal C}^K_1(D_G)$,
$$|A_\rho^\beta f(x)|\le \|f'\|_\infty\Big(|\rho||x|1_{x\in D_G}+|\beta|)+\int_{D_G}\tau_G(u)du\Big)<C<+\infty.$$
It follows that there exists $\eta>1$ such that
 $$\sup_{n\ge1}\ES[|R_n|^\eta]<+\infty.$$ This implies in particular that $(R_n)$ is uniformly integrable.  For a uniformly integrable sequence, convergence in probability implies convergence in $L^1$. One deduces that one only needs to prove that $R_n\rightarrow0$ in probability.\\

\noindent We decompose the increment $f({ X}_{n+1})-f({ X}_n)$ as follows:
\begin{equation}\label{decomp}
f({ X}_{n+1})-f({ X}_n)=f({ X}_{n,1})-f({ X}_n)+f({ X}_{n+1})-f({ X}_{n,1}),
\end{equation} where
${ X}_{n,1}={ X}_{n}+\gamma_{n+1}(\rho_{n+1}{ X}_{n}+\beta_{n+1}).$
First, by the Taylor formula,
\begin{equation}\label{taylor1}
f({ X}_{n,1})-f({ X}_n)=\gam
f'(c_{n+1})(\rho_{n+1}{ X}_{n}+\beta_{n+1}).
\end{equation}
with $c_{n+1}\in[{ X}_{n},{ X}_{n,1}]$. Denoting by ${\cal L}_1$, the operator defined for every $f\in{\cal C}_K^1(\ER)$ by
${\cal L}_1f(x)=f'(x)(\rho x+\beta)$, one deduces that
$$f({ X}_{n,1})-f({ X}_n)=\gam{\cal L}_1
f({ X}_{n})+\gam R_{n,1},$$ with
$$R_{n,1}=f'({ X}_{n})\big((\rho_{n+1}-\rho){ X}_{n}+\beta_{n+1}-\beta\big)
+\big(f'(c_{n+1})-f'(X_n)\big)(\rho_{n+1}{ X}_{n}+\beta_{n+1}).$$
Using that $f$ has compact support, that $(\rho_n)_{n\ge 1}$ and
$(\beta_n)_{n\ge 1}$ are bounded and that
$\gamma_{n}\xrightarrow{n\rightarrow+\infty} 0$, one checks that there exists
$M>0$ and $n_0\ge 1$ such that for every $x$ with $|x|\ge M$, for every
$\theta\in[0,1]$ and $n\ge n_0$,
\begin{equation}\label{f2cn}
f^{'}\left(x+\theta\gamma_n(\rho_n x+\beta_n)\right)=f'(x)=0.
\end{equation}
Furthermore, 
$$\sup_{|x|\in[-M,M]}(|\gamma_n(\rho_n x+\beta_n)|)\xrightarrow{n\rightarrow+\infty}0.$$
Therefore, since $f'$ is uniformly continuous with compact support, for every $\varepsilon>0$, there exists $n_\varepsilon\in\EN$ such that 
$$|R_{n,1}|\le
C\big(|\rho_{n+1}-\rho|+|\beta_{n+1}-\beta|+\varepsilon \big).$$
Using that $\rho_n\rightarrow\rho$ and $\beta_n\rightarrow\beta$, it follows that
$$R_{n,1}\xrightarrow{n\rightarrow+\infty}0\quad a.s.$$
Let us now focus on the second term of the right-hand member  of \eqref{decomp}. We decompose it
as follows:
\begin{align*}
 &f({ X}_{n+1})-f({ X}_{n,1})=\Delta_{n+1,1}-\Delta_{n+1,2}+\Delta_{n+1,3}\qquad\textnormal{with,}\\
&\Delta_{n+1,1}=f\left({ X}_{n}+\Big(\frac{\xi_{n+1}}{a_n}-{ X}_{n}-\frac{b_n}{a_n}\Big)_+\right)-f({ X}_{n}),\\
&\Delta_{n+1,2}=\Big(f({ X}_{n,1})-f({ X}_{n})\Big)
1_{\{\frac{\xi_{n+1}}{a_n}-X_n-\frac{b_n}{a_n}>0\}}\quad\textnormal{and,}\\
&\Delta_{n+1,3}=\left(f({ X}_{n+1})-f\left({ X}_n+
\Big(\frac{\xi_{n+1}}{a_n}-{ X}_{n}-\frac{b_n}{a_n}\Big)_+\right)\right)
1_{\{\frac{\xi_{n+1}}{a_n}-{ X}_{n}-\frac{b_n}{a_n}>0\}}.
\end{align*}
Denote by ${\cal L}_2$, the operator defined on ${\cal C}_K^1(D_G)$ by 
\begin{align}
{\cal L}_2f(x)&=\int_0^{+\infty}\Big(f(x+y)-f(x)\Big)\varphi_G(x+y)dy\nonumber\\
&=\int_0^{+\infty} f'(x+y)\tau_G(x+y)dy.\label{repl2}
\end{align}
In order to prove the lemma, it suffices now to show the three following points:
\begin{align*}
&\textnormal{(i)}\quad \ES[\Delta_{n+1,1}/{\cal F}_n]=\gam
{\cal L}_2f({ X}_{n})+\gam R_{n,2} \quad\textnormal{with}\quad
R_{n,2}\overset{\PE}{\longrightarrow} 0.
\\
&\textnormal{(ii)}\quad R_{n,3}:=\frac{1}{\gam}\ES[\Delta_{n+1,2}/{\cal
F}_n]\overset{\PE}{\longrightarrow}  0,\\
&\textnormal{(iii)} \quad R_{n,4}:=\frac{1}{\gam}\ES[\Delta_{n+1,3}/{\cal
F}_n]\overset{\PE}{\longrightarrow}  0.
\end{align*}
(i)  First, since $\xi_{n+1}$ is independent of ${\cal F}_n$, we have
\begin{equation*}
\ES[\Delta_{n+1,1}/{\cal
F}_n]=\psi_n(X_n)\quad \textnormal{with}\quad\psi_n(x)=\ES\left[f\left(x+\Big(\frac{\xi_{n+1}}{a_n}-x-\frac{b_n}{a_n}\Big)_+\right)-f(x)\right]
\end{equation*}
and a simple transformation yields
\begin{align*}
\psi_n(x)&=\int_0^{+\infty}f'(x+u)\PE\left(\frac{\xi_{n+1}}{a_n}-x-\frac{b_n}{a_n}>u\right)du\\
&=\int_0^{+\infty}f'(x+u)\left(1-F(a_n(x+u)+b_n)\right)du.
\end{align*}
Hence, it follows from \eqref{repl2} that
\begin{align*} &\ES[\Delta_{n+1,1}/{\cal
F}_n]=\gam {\cal L}_2f({ X}_n)+\gam R_{n,2}\quad\textnormal{where,}\quad R_{n,2}=\int_0^{+\infty}H_n(X_n,u)du,\\
&\textnormal{with}\quad H_n(x,u)=f'(x+u)\Big(\frac{1-F(a_n(x+u)+b_n)}{1-F(\theta_n)}-\tau_G(x+u)\Big).
\end{align*}
Since $f\in{\cal C}_K^1(D_G)$, there exists a compact subset $K_G$ of $D_G$ such that $f'(x)=0$ when $x\in K_G^c$. Then,
 $H_n(x,u)=0$  as soon as  $x+u\in K_G^c$. \\
Therefore, one derives from Lemma \ref{lemmeuniform} that for every $\varepsilon>0$, there exists $n_0\in \EN$ such that for every $n\ge n_0$, $a.s.$,
 $$|H_n({ X}_n,u)|\le \varepsilon 1_{\{({ X}_n+u) \in K_G\}}\|f'\|_\infty\quad \Longrightarrow\quad |R_{n,2}|\le  \varepsilon \lambda(K_G)\|f'\|_\infty,$$
where $\|f'\|_\infty=\sup_{\{x\in K_G\}}|f'(x)|$ and $\lambda$ denotes the Lebesgue measure. Hence, $R_{n,2}\xrightarrow{n\rightarrow+\infty}0$ $a.s.$ \\
\noindent (ii) First,  since $f$ has compact support, a similar argument as that used for  \eqref{f2cn} shows that 
there exists $M>0$ and $n_0\in\EN$ such that for every $x$ satisfying $|x|>M$, for every
$\theta\in[0,1]$ and $n\ge n_0$,
$$f^{'}\left(x+\theta\gamma_n(\rho_n x+\beta_n)\right)=0.$$
Therefore, since $f$ is a Lispchitz continuous function and the sequences $(\rho_n)$ and $(\beta_n)$ are bounded, one obtains that
$$|f({ X}_{n,1})-f({ X}_n)|\le C\gam 1_{\{|{ X}_n |\le M \}}\quad \forall n\ge n_0.$$
Hence, 
\begin{align}
|R_{n,3}|\le C\PE\Big(\frac{\xi_{n+1}}{a_n}&-X_n-\frac{b_n}{a_n}>0/{\cal
 F}_n\Big)=C\left(1-F(a_n{ X}_n+b_n)\right)\nonumber\\
 &=C\left(1-F(\max(\xi_1,\ldots,\xi_n)\right)\xrightarrow{n\rightarrow+\infty}0\quad a.s.\label{rn3preuve}
 \end{align}
\noindent (iii) We consider $R_{n,4}$. Since $f$ is a bounded Lipschitz continuous function,
\begin{align*}
\big|f({ X}_{n+1})&-f({ X}_n+\big(\frac{\xi_{n+1}}{a_n}-{ X}_{n}-\frac{b_n}{a_n}\big)_+)\big|
1_{\{\frac{\xi_{n+1}}{a_n}-{ X}_{n}-\frac{b_n}{a_n}>0\}}\\&\le
C\gamma_n\big|(\beta_n+\rho_n{ X}_n)+\rho_n(\frac{\xi_{n+1}}{a_n}-{ X}_{n}-
\frac{b_n}{a_n})_+\big|1_{\{\frac{\xi_{n+1}}{a_n}-{ X}_{n}-\frac{b_n}{a_n}>0\}}.
\end{align*}
It follows that
$$|R_{n,4}|\le C(1+|{ X}_n|)\big(1-F(a_n{ X}_n+b_n)\big)+C\ES\left[(\frac{\xi_{n+1}}{a_n}-{ X}_{n}-
\frac{b_n}{a_n})_+/{\cal F}_n\right].$$
On the one hand,  Lemmas
\ref{lemme2} and \ref{lemme3} imply that the second term of the right-hand member converges to 0 in
$L_1$. On the other hand, for the first term, we use the Skorokhod Representation Theorem.
By Proposition \ref{rappelgne}, $({ X}_n)_{n\ge 1}$ converges weakly to a random variable $X_\infty$ with distribution $\nu_G$. Hence, by the Skorokhod Representation Theorem, one can construct $(\tilde{X}_n)$ and $\tilde{X}_\infty$ on a probability space $(\tilde{\Omega},\tilde{\cal F},\tilde{P})$,
 such that ${\cal L}(\tilde{X}_n)={\cal L}({ X}_n)$, ${\cal L}(\tilde{X}_\infty)={\cal L}({ X}_\infty)$
 and   $\tilde{X}_n\xrightarrow{n\rightarrow+\infty} \tilde{X}_\infty$ $a.s.$ Since ${\rm supp}(\nu_G)\subset D_G$, it implies that  $a.s.$, there exists $n_0(\omega)$ such that  $(\tilde{X}_n(\omega))_{n\ge n_0(\omega)}$ is contained in a compact set $K_G(\omega)$ of $D_G$. Hence, using Lemma \ref{lemmeuniform} and the fact that $(\tilde{X}_n(\omega))$ is bounded, one obtains: 
$$(1+|\tilde{X}_n|)\big(1-F(a_n\tilde{X}_n+b_n)\big)\xrightarrow{n\rightarrow+\infty}
0\quad a.s.$$ 
It follows that
$$(1+|{ X}_n|)\big(1-F(a_n{ X}_n+b_n)\big)\overset{\PE}{\longrightarrow}0 \quad\textnormal{as $n\rightarrow+\infty$}.$$
This completes the proof. 
\end{proof}
\begin{prop}\label{propvalda}
Let $F$ be of type 1, $(2,\alpha)$ with $\alpha>2$ or $(3,\alpha)$ and let $(a_n,b_n)_{n\ge1}\in{\cal C}(F)$ such that $(\rho_n)$ and $(\beta_n)$ converge to some finite numbers $\rho$ and $\beta$ respectively. Assume $\mathbf{H_\Lambda^2}(F)$ if $F$ is of type 1.  Suppose that $({\cal X}^{(n)})$ admits at least one weak limit denoted by ${\cal X}^{\infty}$. Then, ${\cal X}^{\infty}$ solves the martingale problem $({\cal A}_\rho^\beta,\nu_G,{\cal C}_K^1(D_G))$. Moreover, if existence and uniqueness hold for this martingale problem,  ${\cal X}^{\infty}$ is a stationary Markov process.
\end{prop}
\begin{proof} First, by Lemma \ref{lemme1}, we can choose $\delta<x_F$ such that $F(x)=0$ for every $x\le \delta$ and such that
the assertions of Lemmas \ref{lemme2} and \ref{lemme3} hold. As a consequence, we can assume in the sequel that the conclusions of Lemma \ref{lemme5} hold.\\ 
Let $f\in{\cal C}_K^1(D_G)$. By  Lemma \ref{lemme5}, for every $n\ge
2$,
$$\sum_{k=1}^n\ES[f(X_k)-f(X_{k-1})/{\cal F}_{k-1}]=\sum_{k=2}^n\gamma_k
{\cal A}_\rho^\beta f({ X}_{k-1})+\sum_{k=2}^n\gamma_k R_{k-1}.$$
where $(R_n)_{n\ge 1}$ is an $({\cal F}_n)$-adapted sequence such that $R_n\xrightarrow{n\rightarrow+\infty}0$
in $L^1$.
Denoting by $(M_n)_{n\ge0}$ the martingale defined by
$$M_n=\sum_{k=1}^n\left(f(X_k)-f(X_{k-1})-\ES[f(X_k)-f(X_{k-1})/{\cal F}_{k-1}]\right),$$
one obtains that
$$f({ X}_n)=f({ X}_1)+\sum_{k=2}^n\gamma_k
{\cal A}_\rho^\beta f({ X}_{k-1})+\sum_{k=2}^n\gamma_k R_{k-1}+M_n.$$ 
It follows that for every $n\ge 2$,
$$f({\cal X}_t^{(n)})-f({\cal X}_0^{(n)})=M_t^{(n)}+\sum_{k=n+1}^{N(n,t)}\gamma_k({\cal A}_\rho^\beta f({ X}_{k-1})+R_{k-1}),$$
where $(M^{(n)})$ is the sequence of  martingales   defined
by $M^{(n)}_t=M_{N(n,t)}-M_{n}$. One observes that
$$\sum_{k=n+1}^{N(n,t)}\gamma_k {\cal A}_\rho^\beta f({ X}_{k-1})=\int_0^t{\cal A}_\rho^\beta f({\cal X}_s^{(n)})ds+ \big(\sum_{k=n+1}^{N(n,t)}\gamma_k
-t\big){\cal A}_\rho^\beta f({\cal X}_t^{(n)}).$$  
Hence,
\begin{equation}\label{prob}
f({\cal X}_t^{(n)})-f({\cal X}_0^{(n)})-\int_0^t{\cal A}_\rho^\beta f({\cal X}_s^{(n)})ds=M_t^{(n)}+\tilde{R}_t^{(n)},
\end{equation}
where 
$$\tilde{R}_t^{(n)}=\big(\sum_{k=n+1}^{N(n,t)}\gamma_k
-t\big){\cal A}_\rho^\beta f({\cal X}_t^{(n)})+\sum_{k=n+1}^{N(n,t)}\gamma_k R_{k-1}$$ 
First, using that $f$ belongs to ${\cal C}_K^1(D_G)$ and \eqref{repl2}, one checks that ${\cal A}_\rho^\beta$ is a bounded uniformly continuous function.
Then, using Theorem VI.1.14 of \cite{jacodshiryaev}, it follows that $H^f:\mathbb{D}(\ER_+,\ER)\rightarrow\mathbb{D}(\ER_+,\ER)$ defined by  
$$(H^f(\alpha))_t=f(\alpha_t)-f(\alpha_0)-\int_0^t {\cal A}_\rho^\beta f(\alpha_s)ds$$
is continuous for the Skorokhod topology. Therefore, if $({\cal X}^{(n_k)})$ denotes a convergent subsequence and ${\cal X}^{\infty}$ its limit, $H^f({\cal X}^{(n_k)})$  and $(M^{(n_k)}+\tilde{R}^{(n_k)})$ converge weakly to $H^f({\cal X}^\infty)$.\\
Second, by the definition of $N(n,t)$ and the fact that $(\gamma_n)$ decreases, $$\big(\sum_{k=n+1}^{N(n,t)}\gamma_k
-t\big)\le \gamma_{N(n,T)+1}\le \gamma_n.$$ 
Then, since $A_\rho^\beta f$ is bounded and $R_n\rightarrow0$ in $L^1$, we obtain that
for every positive $T$,
$$\ES[\sup_{t\in[0,T]}|\tilde{R}_t^{(n)}|]\le C\gamma_n + T\sup_{k\ge n} \ES[|R_k|]\xrightarrow{n\rightarrow+\infty}0.\quad $$
Therefore, it follows from Lemma VI.3.31 of \cite{jacodshiryaev} that $(M^{(n_k)})$ converges weakly to $H^f({\cal X}^\infty)$.
Using that $(M^{(n_k)})$ is a sequence of martingales with uniformly bounded jumps (since $f$ is bounded), we derive from Corollary IX.1.19 of \cite{jacodshiryaev} that $H^f({\cal X}^{\infty})$ is a local martingale for every $f\in{\cal C}_K^1(D_G)$. Furthermore, since $f$ and $A_\rho^\beta f$ are bounded, $H^f$ is a true martingale. Hence, ${\cal X}^\infty$ solves the martingale problem $({\cal A}_\rho^\beta,\mu,{\cal C}_K^1(D_G))$ where $\mu={\cal L}({\cal X}^\infty_0)$. Furthermore, since for every $t$,
 $({\cal X}^{(n)}_t)_{n\ge1}$ is a subsequence of $({ X}_n)_{n\ge 1}$,
 one derives from Proposition \ref{rappelgne} that ${\cal X}^\infty_t$ has distribution  $\nu_G$
  for every $t$. This involves that, on the one hand, ${\cal L}({\cal X}^{\infty}_0)=\nu_G$ and that, on the other hand,
  ${\cal L}({\cal X}^\infty_t)={\cal L}({\cal X}^\infty_0)$ for every $t\ge0$. Therefore, if existence and uniqueness hold for the martingale problem $(A_\rho^{\beta},\nu_G,{\cal C}_K^1(D_G))$, then ${\cal X}^{\infty}$ is a Markov process that is stationary since it satisfies ${\cal L}({\cal X}^{\infty}_t)={\cal L}({\cal X}_0^{\infty}).$
 \end{proof}
\section{Proof of Theorem \ref{maxprincipal}}\label{resume}
By Proposition \ref{tight}, we know that under the assumptions of Theorem \ref{maxprincipal}, $({\cal X}^{(n)})$  is tight and then admits at least one weak limit. Denoting by ${\cal X}^{\infty}$ one of these weak limits, we derive from Proposition \ref{propvalda} that   if $(a_n,b_n)$ is a sequence that belongs to ${\cal C}(F)$  such that $(\rho_n)$ and $(\beta_n)$ converge to some finite numbers $\rho$ and $\beta$, ${\cal X}^{\infty}$ solves the martingale problem $({\cal A}_\rho^\beta,\nu_G,{\cal C}_K^1(D_G))$. In particular, existence holds for this martingale problem under the preceding assumptions on the sequence $(a_n,b_n)$. Then, it  also follows from Proposition \ref{propvalda} that if moreover, uniqueness holds for this martingale problem, $({\cal X}^{(n)})$ converges weakly to ${\cal X}^{\infty}$ for the Skorokhod topology on $\mathbb{D}(\ER_+,\ER)$ and that ${\cal X}^{\infty}$ is a stationary Markov process.
Then, the reader can check that proving Theorem \ref{maxprincipal} comes now to
 show the four following assertions:
 \begin{itemize}
\item[(i)] There exists $(a_n,b_n)\in{\cal C}(F)$ such that $(\rho_n)$ and $(\beta_n)$ converge to some finite numbers $\rho$ and $\beta$.
\item[(ii)] $\rho$ and $\beta$ satisfy \eqref{rhobeta}.
\item[(iii)] If Theorem \ref{maxprincipal} holds for a particular $(a_n,b_n)\in{\cal C}(F)$, then it holds for every
$(a_n,b_n)\in{\cal C}(F)$.
\item[(iv)] Uniqueness holds for the martingale problem $({\cal A},\nu_G,{\cal C}_K^1(D_G))$ where ${\cal A}={\cal A}_{\rho_G}^{\beta_G}$.
\end{itemize}
(i) When $F$ is of type 1, this assertion is exactly Assumption $\mathbf{H_\Lambda^1}(F)$. When $F$ is of type $(2,\alpha)$,
we know that $(\theta_n,0)\in{\cal C}(F)$. With this choice, we have first that $\beta_n=0$. Second, since $1-F$ is a non-decreasing regularly varying function with index $-\alpha$, $(1/(1-F))^{\leftarrow}$ is also regularly varying with index $1/\alpha$ (see $e.g.$ \cite{resnick}, Proposition 0.8, p. 23). Then, the Karamata representation yields:
$$\theta_n=c(n)\exp\left(\int_1^{n}\frac{\rho(s)}{s}ds\right)$$
where $c(x)\xrightarrow{x\rightarrow+\infty} c>0$ and $\rho(x)\xrightarrow{x\rightarrow+\infty}1/\alpha$. Set $a_n:=c\exp(\int_1^{n}(\rho(s)/s)ds)$. Since $a_n\sim\theta_n$,
$(a_n,0)\in{\cal C}(F)$. Furthermore,
$$\frac{a_{n-1}}{a_n}=\exp\left(-\int_{n-1}^{n}\frac{\rho(s)}{s}ds\right)=1-\frac{1}{\alpha n}+o(\frac{1}{n})\quad\textnormal{as $n\rightarrow+\infty$}.$$
This involves that $\rho_n\xrightarrow{n\rightarrow+\infty}-1/\alpha$.  Finally, when $F$ is of type $(3,\alpha)$, some very close arguments lead to the existence of a sequence $(a_n,b_n)\in{\cal C}(F)$ such that $\beta_n=0$ for every $n\ge1$ and  $\rho_n\xrightarrow{n\rightarrow+\infty}1/\alpha$.\\
\noindent (ii) If $F$ is of type $(2,\alpha)$ or $(3,\alpha)$, this assertion has been shown in (i). Then, let us suppose that $F$ is of type 1.  By Proposition \ref{propvalda},   $\nu_\Lambda$ is an invariant distribution of ${\cal A}_\rho^\beta$. It follows that 
$$\int {\cal A}_\rho^\beta f(x)\nu_\Lambda(dx)=0\quad\forall f\in{\cal C}_K^1(\ER)\quad \textnormal{(see $e.g.$ \cite{ethier}).}$$
Hence, for every $f\in{\cal C}_K^1(\ER)$, $(\rho,\beta)$ is solution to 
\begin{align}
&a_1(f,\nu_\Lambda)\rho+a_2(f,\nu_\Lambda)\beta+a_3(f,\nu_\Lambda,\Phi_\Lambda)=0\quad\textnormal{with}\label{eqe}\\
&a_1(f,\nu_\Lambda)=\int xf'(x)\nu_\Lambda(dx),\quad a_2(f,\nu_\Lambda)=\int f'(x)\nu_\Lambda(dx) \quad\textnormal{and,}\nonumber\\
&a_3(f,\nu_\Lambda,\Phi_\Lambda)=\int\int_0^{+\infty}\left[f(x+y)-f(x)\right]\Phi_\Lambda(x+y)dy\nu_\Lambda(dx).\nonumber
\end{align}
By Remark \ref{remarkong}, we deduce from some particular cases (exponential and normal distributions) that $(\rho,\beta)=(0,-1)$
is a solution to this system. Then, let us show that this is the only one: the fact that $(\rho,\beta)=(0,-1)$ is a solution to \eqref{eqe} involves  in particular that   $a_3(f,\nu_\Lambda,\Phi_\Lambda)=a_2(f,\nu_\Lambda)$. Hence, Equation \eqref{eqe} can be reduced to:
\begin{equation}\label{eq2par}
a_1(f,\nu_\Lambda)\rho+a_2(f,\nu_\Lambda)(\beta+1)=0\quad\forall f\in{\cal C}_K^1(\ER).
\end{equation}
We now exhibit a function $f$ such that $a_1(f,\nu_\Lambda)\neq0$ and $a_2(f,\nu_\Lambda)=0$. For every $t\ge0$, denote by $f_t$ the function defined by:
\begin{equation*}
f_t(x)=\begin{cases}1+\cos(x)&\textnormal{if $x\in[-\pi,0]$,}\\
2&\textnormal{if $x\in[0,t]$,}\\
1+\cos(x-t)&\textnormal{if $x\in[t,t+\pi]$,}\\
0&\textnormal{otherwise.}
\end{cases}
\end{equation*}
For every $t\ge0$, $f_t$ belongs to ${\cal C}^1_K(\ER)$. First, one observes that $xf'_t(x)\le 0$ for every $x\in\ER$ and that 
$xf'_t(x)<0$ on $(-\pi,0)\cup(t,t+\pi)$. Therefore, $a_1(f_t,\nu_\Lambda)<0$ for every $t\ge0$. Second, on the one hand, one can check that 
$\Lambda'(x)>\Lambda'(-x)$ for every $x>0$. Then, using that $f_0'(x)=-f_0'(-x)$ for every $x\in\ER$, it follows that
$a_2(f_0,\nu_\Lambda)=\int f'_0(x)\Lambda'(x)dx<0.$ On the other hand, since $\Lambda'(x)\rightarrow0$ when $x\rightarrow+\infty$, we derive that $a_2(f_t,\nu_\Lambda)>0$ for sufficiently large $t$. It follows from the continuity of  $t\rightarrow a_2(f_t,\nu_\Lambda)$ that there exists $t_0>0$ such that $a_2(f_{t_0},\nu_\Lambda)=0$ (and  $a_1(f_{t_0},\nu_\Lambda)<0$). We deduce that $\rho=0$ and then, that $\beta=-1$.

\noindent (iii) Let $(a_n,b_n)_{n\ge1}\in{\cal C}(F)$ such that $({\cal X}^{(n)})$ converges in distribution on $\mathbb{D}(\ER_+,\ER)$ to ${\cal X}^{\infty}$ and consider another sequence $(\tilde{a}_n,\tilde{b}_n)_{n\ge0}\in{\cal C}(F)$. Set $\tilde{X}_n=(M_n-\tilde{b}_n)/\tilde{a}_n.$
Since $a_n\sim\tilde{a}_n$ and $b_n-\tilde{b}_n=o(a_n)$ when $n\longrightarrow+\infty$, we have:
$$X_n-\tilde{X}_n=\frac{M_n-b_n}{a_n}\left(1-\frac{a_n}{\tilde{a}_n}\right)+\frac{\tilde{b}_n-b_n}{\tilde{a}_n}=\varepsilon_1(n)X_n+\varepsilon_2(n),$$
where $\varepsilon_1(n)\xrightarrow{n\rightarrow+\infty}0$ and $\varepsilon_2(n)\xrightarrow{n\rightarrow+\infty}0$.
Then, we derive from the tightness of  $({\cal X}^{(n)})$  on  $\mathbb{D}(\ER_+,\ER)$ that, for every $\eta>0$, for every $T>0$,
$$\PE\left(\sup_{k=n}^{N(n,T)}|\varepsilon_1(k)X_k|>\eta\right)\le\PE\left(\sup_{t\in[0,T]}|{\cal X}^{(n)}_t|>\frac{\eta}{\underset{k\ge n}{\inf}|\varepsilon_1(k)|}\right)\xrightarrow{n\rightarrow+\infty}0.$$
Therefore, denoting by $\tilde{\cal X}^{(n)}$ the c\`adl\`ag process defined on $\ER_+$ by $\tilde{\cal X}^{(n)}_t=\tilde{X}_{N(n,t)}$, it follows that for every $\eta>0$, for every $T>0$,
$$\PE(\sup_{t\in[0,T]|}|{\cal X}^{(n)}_t-\tilde{\cal X}^{(n)}_t|>\eta)\xrightarrow{n\rightarrow+\infty}0.$$
Hence, $(\tilde{\cal X}^{(n)})$ converges weakly to ${\cal X}^\infty$ (see $e.g.$ \cite{jacodshiryaev}, Lemma 3.31, p. 352).\\
(iv) First, by a change of variable, one obtains:
$${\cal A}f(x)=(\rho_G
x+\beta_G)f'(x)+\int_0^{+\infty}\Big(f(x+(z-x)_+)-f(x)\Big)\varphi_G(z)dz. $$
Then, ${\cal A}$ can be viewed as the infinitesimal generator associated with the SDE:
\begin{equation}\label{EDS}
dX_t=(\rho_G X_{t^-}+\beta_G)dt+\int (z-X_{t^-})_+\varphi_G(z)N(dt,dz),
\end{equation}
where $N$ is a Poisson random measure with intensity measure $dt\otimes\lambda(dz)$. More precisely, if $(X_t)$ is a solution to
\eqref{EDS} with initial value $X_0$, one can check that $(X_t)$ is a solution to the martingale problem $({\cal A},\mu)$ where $\mu={\cal L}(X_0)$. Now, by Lepeltier and Marchal (\cite{lepeltier}, Corollary II.10 and Theorem II.13), uniqueness for the martingale problem $(A,\delta_{x_0})$ holds if pathwise uniqueness holds for SDE \eqref{EDS} with initial value $x_0$. Therefore, if we want to prove uniqueness for the martingale problem $(A,\nu_G)$, we need only to show that for every $x_0\in\textnormal{supp}(\nu_G)$, pathwise uniqueness holds for the SDE with initial value $x_0$.  \\
Assume first that $F$ is of type 1. Since $u\rightarrow(u)_+$ is a Lipschitz continuous function which is null on $\ER_-$, one obtains that for every $N>0$,
for every $x,y\in[-N,N]$, 
$$\int \left[(z-x)_+-(z-y)_+\right]^2(\varphi_G(z))^2dz\le C|x-y|^2\int_{-N}^{+\infty}\exp(-2z)dz\le C_N|x-y|^2.$$
Therefore, pathwise uniqueness for \eqref{EDS} follows for every $x_0\in\ER$ from \cite{lepeltier} (Theorem III.3).\\
Second, suppose that  $F$ is of type $(2,\alpha)$. On the one hand, $x\rightarrow\rho_G x+\beta_G:=-\alpha^{-1}x$ is a Lipschitz continuous function. On the other hand, for every positive $\varepsilon$ and $N$ such that $0<\varepsilon<N$, for every $x,y\in[\varepsilon,N]$, 
$$\int \left[(z-x)_+-(z-y)_+\right]^2(\varphi_G(z))^2dz\le C|x-y|^2\int_{\varepsilon}^{N}\frac{\alpha}{z^{1+\alpha}}dz\le C_{\varepsilon,N}|x-y|^2.$$
Now, since there are only positive jumps, any solution of \eqref{EDS} with initial value $x_0\in\textnormal{supp}(\nu_G)=(0,+\infty)$ is a positive process.
Therefore, pathwise uniqueness follows again from Theorem III.3 of \cite{lepeltier}. \\
Finally, we assume that $F$ is of type $(3,\alpha)$.  Using that $X_{t^-}+(z-X_{t^-})_+\le z$ and that $\varphi_G(z)=0$ when $z>0$, a similar  argument  yields the result in this case. 

\end{document}